\documentclass[a4paper,12pt,reqno]{amsart}
\usepackage{amssymb}
\usepackage{latexsym}
\usepackage{amsmath}
\usepackage[english]{babel}
\usepackage{amsfonts}
\usepackage{enumerate}
\usepackage[all]{xy}
\usepackage{verbatim}
\usepackage{epic,eepic}
\CompileMatrices

\newlength{\widthuparrow}
\newcommand{\nuparrow}{\
\raisebox{-2.5pt}{\makebox[\widthuparrow][l]{\makebox[0pt]{{\reflectbox{\begin{rotate}{90}%
\ensuremath{\nrightarrow}\end{rotate}}}}}}\ }

\newcommand{\ind}{\operatorname{ind}}
\newcommand{\ev}{\operatorname{ev}}
\newcommand{\lbr}{\begin{bmatrix}}
\newcommand{\rbr}{\end{bmatrix}}
\newcommand{\Ker}{\operatorname{Ker}}
\newcommand{\I}{\operatorname{Im}}
\newcommand{\im}{\operatorname{im}}\newcommand{\Coker}{\operatorname{Coker}}
\newcommand{\Hom}{\operatorname{Hom}}
\newcommand{\Ext}{\operatorname{Ext}}
\newcommand{\Tor}{\operatorname{Tor}}
\newcommand{\Ind}{\operatorname{Ind}}
\newcommand{\Res}{\operatorname{Res}}
\newcommand{\Lie}{\operatorname{Lie}}
\newcommand{\End}{\operatorname{End}}
\newcommand{\rd}{\operatorname{d}}\newcommand{\SM}{\operatorname{SM}}\newcommand{\Stab}{\operatorname{Stab}}
\newcommand{\St}{\operatorname{St}}
\newcommand{\soc}{\operatorname{soc}}
\newcommand{\rad}{\operatorname{rad}}
\newcommand{\hd}{\operatorname{hd}}
\newcommand{\cha}{\operatorname{char}}
\newcommand{\ch}{\operatorname{ch}}
\newcommand{\pr}{\operatorname{pr}}
\newcommand{\Di}{\operatorname{Div}}
\newcommand{\height}{\operatorname{ht}}
\newcommand{\spa}{\operatorname{span}}
\newcommand{\sign}{\operatorname{sign}}
\usepackage{bbm} 
\newcommand\cA{\mathcal A}
\newcommand\cC{\mathcal C}
\newcommand\C{\mathbb C}
\newcommand\Z{\mathbb Z}
\newcommand\R{\mathbb R}
\newcommand\X{\mathbb X}
\newcommand\N{\mathbb N}
\newcommand\Q{\mathbb Q}
\newcommand\hb{\mathbb H}
\newcommand\F{\mathbb F}
\newcommand\g{\mathfrak g}
\newcommand\gt{{\tilde{g}}}
\newcommand\K{\mathcal K}
\newcommand\cN{\mathcal N}
\newcommand\cQ{\mathcal Q}
\newcommand\cc{\mathcal C}
\newcommand\mo{\operatorname{mod}}

\newcommand\pf{\noindent {\bf Proof:  }}
\newcommand\lem{\noindent {\bf Lemma.  }}
\newcommand\exer{\noindent {\bf Exercise.  }}
\newcommand\notation{\noindent {\bf Notation.  }}
\newcommand\remark{\noindent {\bf Remark.  }}
\newcommand\iu{{\underline{i}}}
\newcommand\ju{{\underline{j}}}
\newcommand\au{{\underline{a}}}
\newcommand{\hni}{H_i^{0}}

\usepackage{amsthm}
\newtheorem{thm}{Theorem}[section]

\newtheorem{prop}[thm]{Proposition}

\theoremstyle{definition}

\numberwithin{equation}{section}

\title{Cohomology of line bundles on the flag variety for type $G_2$
}
\author{Henning Haahr Andersen and 
Kaneda Masaharu }

\address{HHA: Centre for Quantum Geometry of Moduli Spaces, Aarhus University,
8000  Aarhus C, Denmark}
\email{mathha@qgm.au.dk}

\address{KM: Department of Mathematics, Osaka City University,
Osaka 558-8585 Japan}
\email{kaneda@sci.osaka-cu.ac.jp}
\thanks
{KM supported in part
by JSPS Grant in Aid
for Scientific Research \#23540023}

\begin{document}

\maketitle
\begin{abstract}
In the case of a 
simple algebraic group $G$ of type $G_2$ over a field of characteristic $p>0$ we study the cohomology modules of line bundles on the 
flag variety for $G$. Our main result is a complete determination of the vanishing behavior of such cohomology in the case where the line bundles in question are induced by
characters 
from the lowest $p^2$-alcoves.

When $U_q$ is the quantum group corresponding to $G$ whose parameter $q$ is a complex root of unity of order prime to $6$ we give a complete (i.e. covering all characters) 
description of the vanishing behavior for the corresponding quantized cohomology modules.
\end{abstract}
 
\section{Introduction}
Throughout this paper $G$ will denote an almost simple algebraic group of type $G_2$ over a field of characteristic $p>5$. We shall give a complete description of the 
vanishing behavior of the cohomology of those line bundles on the flag variety for $G$ which correspond to weights from the lowest $p^2$-alcoves in all the Weyl chambers
for $G$. To achieve this we shall need almost all available methods that we know for such computations. 

The strategy we use will apply to other types as well but a similar description for higher ranks seems out of reach. Our computations also give some information about the 
$G$-structure of the
cohomology modules but the full story here is still open. By an illustrating example (see Section 8) we demonstrate that our present techniques are not enough to handle the case where the
weights lie outside the $p^2$-alcoves. On the other hand, we can handle all weights for the corresponding quantized $G_2$ situation at complex roots of unity, see below and
Section 9.

Our assumption $p>5$ is necessary in order for alcoves to contain integral weights in their interior. We need this in order to apply translation functors effectively. However,
many parts of our methods apply for $p\leq 5$ as well but we leave it to the reader to formulate and check the statements in these cases (for instance when $p=2$ the only weight
in the lowest $p^2$-range is $-\rho$ so there is nothing to check in that case). Our results are given in terms of figures
of alcoves decorated with those numbers $i$ for which there is an extra non-vanishing $i$-th cohomology of the line bundles corresponding to the weights of the alcove. To see the
patterns in these figures we need to take $p$ a bit larger,
$p\geq17$, although our proofs work for all $p>5$! 
The figures for smaller $p$ will be a subset of our figures. We leave 
to the reader to do the appropriate cuts.

Let $U_q$ denote the quantum group of type $G_2$ with $q$ a complex root of unity of order $l$. We assume that $l$ is prime to $6$. Then there are cohomology modules for $U_q$
(derived functors of the induction functor from a Borel subalgebra of $U_q$) which are quantized counterparts of the above line bundle cohomology. Our techniques apply to this
case as well. Moreover, the modular phenomena of `higher alcoves' (for the powers of $p$) are not present in this case. Hence we are able to give the complete vanishing behavoir
of quantized line bundle cohomology for type $G_2$.

The questions studied in this paper go back to the final section of \cite{A81}. Here the first author sketched a description of the vanishing behavior of line bundle cohomology on
the
flag varieties for all rank $2$ groups. However, as pointed out by the second author there are some errors in the statements and J. Humphreys pinpointed in \cite{H} some specific
inaccuracies in the $G_2$-case. In the survey article \cite{A07} some of this was repaired but when we needed this kind of results in connection with our work \cite{AK10} we
decided to carefully go through all computations. The outcome of this was recorded in Appendix B of \cite{AK10} where we referred to a preliminary version of the present paper
for details. With a few modifications we follow the 6 step program outlined in this appendix.
 
 We are grateful to J. E. Humphreys for many discussions over the years on these issues and for some detailed comments on the contents of the preliminary version of this paper.

\begin{figure}
\begin{center}
{\tiny%
  \setlength{\unitlength}{3.6mm}
  \begin{picture}(40,44)
\put(22.5,40.6){\Large\makebox(0,0){$e$}}
\put(31.5,35.6){\Large\makebox(0,0){$t$}}
\put(12.5,40.6){\Large\makebox(0,0){$s$}}
\put(4.5,35.6){\Large\makebox(0,0){$x=s
t$}}
\put(31,25.2){\Large\makebox(0,0){$y=t
s$}}
\put(4.5,25.2){\Large\makebox(0,0){$z=
s
t
s$}}
\put(30.5,18.2){\Large\makebox(0,0){$w=
t
s
t$}}
\put(30.5,7.2){\Large\makebox(0,0){$
t
z$}}

\put(18.05,23.6){\makebox(0,0){1}}
\put(18.6,24.7){\makebox(0,0){2}}
\put(18.2,25.6){\makebox(0,0){3}}
\put(18.2,26.6){\makebox(0,0){4}}
\put(18.6,27.6){\makebox(0,0){5}}
\put(18.05,28.6){\makebox(0,0){6}}
\put(19.6,27.9){\makebox(0,0){7}}
\put(19,29.1){\makebox(0,0){8}}
\put(19.4,30.0){\makebox(0,0){11}}
\put(19.4,31){\makebox(0,0){13}}
 \put(18.9,32.0){\makebox(0,0){15}}
 \put(19.5,33.0){\makebox(0,0){16}} 
\put(18.05,32.5){\makebox(0,0){19}}
\put(18.5,33.4){\makebox(0,0){21}}
\put(18.2,34.4){\makebox(0,0){23}}   
\put(18.2,35.3){\makebox(0,0){26}}
\put(18.5,36.3){\makebox(0,0){28}}
\put(19.5,36.8){\makebox(0,0){30}}

\put(20.4,27.9){\makebox(0,0){9}}
\put(21.0,29.0){\makebox(0,0){10}}
\put(20.5,30){\makebox(0,0){12}}
\put(20.5,31){\makebox(0,0){14}}
\put(21.2,31.8){\makebox(0,0){17}}
 \put(20.5,32.8){\makebox(0,0){18}}
  \put(22,32.5){\makebox(0,0){20}}
  \put(21.5,33.4){\makebox(0,0){22}}
  \put(21.9,34.4){\makebox(0,0){25}}
  \put(21.9,35.3){\makebox(0,0){27}}
  \put(21.5,36.3){\makebox(0,0){29}}   
  
 \put(23,32.5){\makebox(0,0){24}}
  \put(23.5,33.4){\makebox(0,0){31}}
  \put(23,34.4){\makebox(0,0){32}}
  \put(23,35.3){\makebox(0,0){33}}
  \put(23.5,36.3){\makebox(0,0){34}}    
  \put(23.0,37.3){\makebox(0,0){35}}

  \multiput(0,0)(5,0){7}{\line(4,7){5}}%
  \multiput(0,8.7)(5,0){7}{\line(4,7){5}}%
  \multiput(0,17.4)(5,0){7}{\line(4,7){5}}%
  \multiput(0,26.2)(5,0){7}{\line(4,7){5}}%
  \multiput(0,34.8)(5,0){7}{\line(4,7){5}}%
  \multiput(5,0)(5,0){7}{\line(-4,7){5}}%
  \multiput(5,8.7)(5,0){7}{\line(-4,7){5}}%
  \multiput(5,17.4)(5,0){7}{\line(-4,7){5}}%
  \multiput(5,26.2)(5,0){7}{\line(-4,7){5}}%
  \multiput(5,34.8)(5,0){7}{\line(-4,7){5}}%
  \multiput(0,0)(0,2.9){4}{\line(7,4){5}}%
 \multiput(0,14.6)(0,2.9){4}{\line(7,4){5}}\multiput(0,23.3)(0,2.90){7}{\line(7,4){5}}
 \multiput(5.0,0)(0,2.89){6}{\line(7,4){5}}%
 \multiput(5.0,17.5)(0,2.89){9}{\line(7,4){5}} \multiput(10,0)(0,2.89){6}{\line(7,4){5}}%
 \multiput(10,20.4)(0,2.89){8}{\line(7,4){5}}
\multiput(15,0)(0,2.89){7}{\line(7,4){5}}%
 \multiput(15,23.3)(0,2.89){7}{\line(7,4){5}}
\multiput(20,0)(0,2.89){8}{\line(7,4){5}}%
 \multiput(20,26.1)(0,2.89){6}{\line(7,4){5}}
 \multiput(25,0)(0,2.89){9}{\line(7,4){5}}%
 \multiput(25,28.9)(0,2.89){5}{\line(7,4){5}}
\multiput(30,0)(0,2.89){10}{\line(7,4){5}}%
 \multiput(30,31.8)(0,2.89){4}{\line(7,4){5}}
  \multiput(4.9,0)(0,2.91){10}{\line(-7,4){5}}%
  \multiput(4.9,31.9)(0,2.91){4}{\line(-7,4){5}}
  \multiput(9.9,0)(0,2.91){9}{\line(-7,4){5}}%
  \multiput(9.9,29.0)(0,2.91){5}{\line(-7,4){5}}
  \multiput(14.9,0)(0,2.91){8}{\line(-7,4){5}}%
  \multiput(14.9,26.1)(0,2.91){6}{\line(-7,4){5}}
  \multiput(20,0)(0,2.91){7}{\line(-7,4){5}}%
  \multiput(20,23.2)(0,2.91){7}{\line(-7,4){5}}
  \multiput(25,0)(0,2.91){6}{\line(-7,4){5}}%
  \multiput(25,20.3)(0,2.91){8}{\line(-7,4){5}}
  \multiput(30,0)(0,2.91){5}{\line(-7,4){5}}%
  \multiput(30,17.4)(0,2.91){9}{\line(-7,4){5}}
  \multiput(35,0)(0,2.91){4}{\line(-7,4){5}}%
  \multiput(35,14.5)(0,2.91){10}{\line(-7,4){5}}
 
\allinethickness{3pt}
  \put(20,21.7){\line(1,0){15}}
  \put(20,21.7){\line(-1,0){20}}
  \put(17.5,21){\line(0,1){22.5}}
  \put(17.5,21){\line(0,-1){21}}
  \put(17.5,21.7){\line(4,7){12.4}}
  \put(17.5,21.7){\line(-4,-7){12.4}}
  \put(17.5,21.7){\line(-4,7){12.4}}
  \put(17.5,21.7){\line(4,-7){12.4}}
  \put(17.5,21.8){\line(7,4){17.5}}
  \put(17.5,21.7){\line(-7,-4){17.5}}
  \put(17.5,21.7){\line(7,-4){17.5}}
  \put(17.5,21.8){\line(-7,4){17.5}}


  \linethickness{0.075mm}
  \multiput(0,0)(2.5,0){15}{\line(0,1){43.4}}
  \multiput(0,0)(0,4.35){11}{\line(1,0){35}}
\end{picture}}

\caption{Label of some alcoves
and Weyl chambers}
\end{center}
\end{figure}
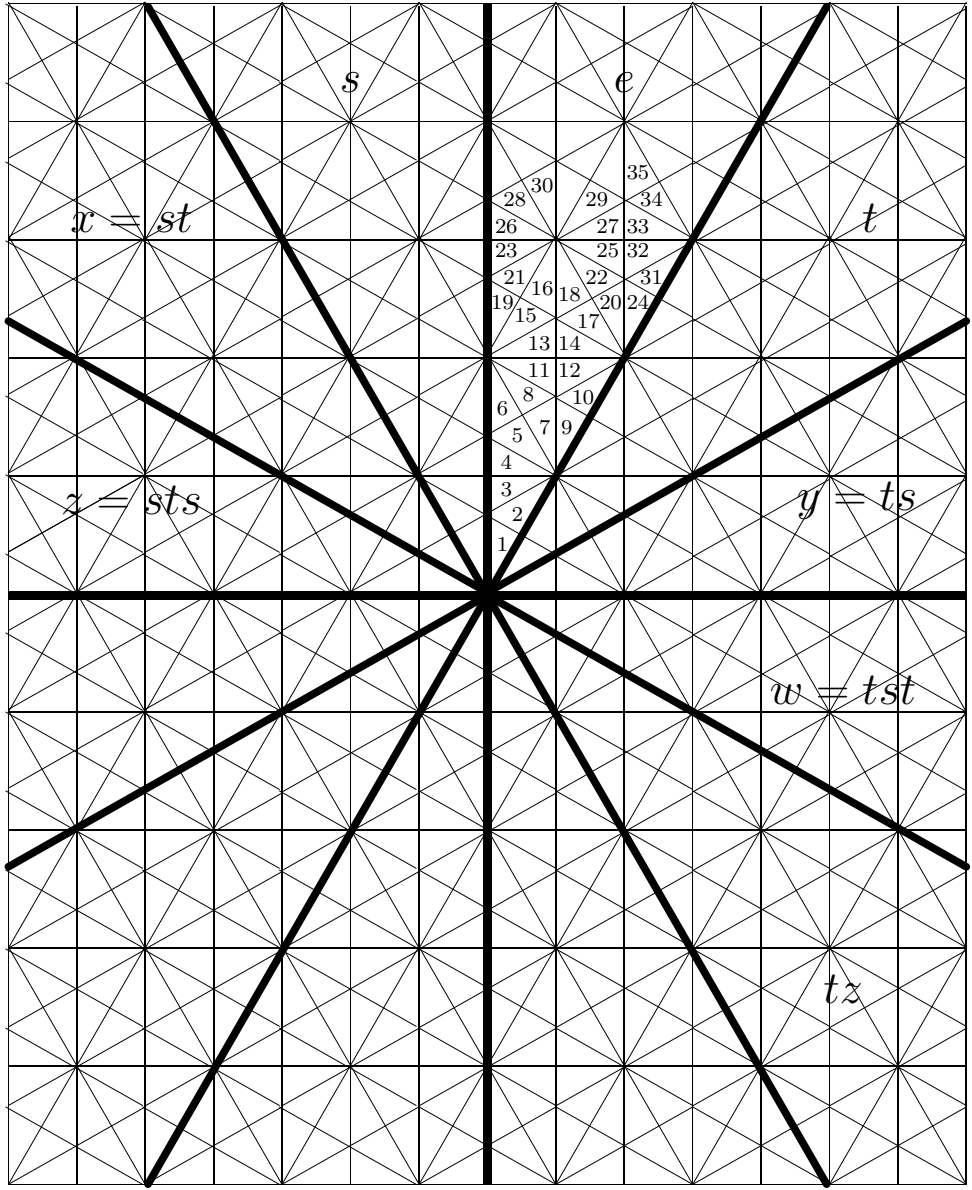

\section{Notation and preliminaries}

We shall generally use the notation from \cite{AK10}. In this section we recall the main players and we introduce some specific notations for the $G_2$-case we are dealing with.
Bits of further notations are introduced as we go along and we 
refer to loc. cit. for any unexplained notation.

In the algebraic group $G$ 
of type $G_2$
we fix a maximal torus $T$ and we let $R$ denote the corresponding root system. We choose a set of positive roots $R^+$ and 
$\alpha$ and $\beta$ with $\beta$ long denote the two simple roots in $R^+$. Then $B$ will denote the Borel subgroup corresponding to $-R^+$. The line bundles on the flag variety
$G/B$ are induced by the characters of $B$. We let $X$ denote the character group of $B$, and for $\lambda \in X$ we use the abbreviation  $H^i(\lambda)$ for the $i$-th cohomology
group of the line bundle on $G/B$ induced by $\lambda$.

Inside $X$ we have the set of dominant weights $X^+ = \{\lambda \in X \mid \langle \lambda, \gamma^{\vee} \rangle \geq 0 \text { for all } \gamma \in R^+ \}$. The Weyl group $W$
is the group generated by the reflections in the walls of $X^+$, i.e., by the two elements $s = s_\alpha$ and $t = s_\beta$. The bottom alcove
in $X^+$ is $A_1 = \{ \lambda \in X^+ \mid \langle \lambda + \rho, \gamma^{\vee} \rangle < p \text { for all } \gamma \in R^+ \}$. Here $\rho = 5 \alpha + 3 \beta$ is half the sum
of the positive roots. Note that $0 \in A_1$ because of our assumption
$p > 5$. The affine Weyl group $W_p$ is the group generated by the reflections in the $3$ walls of $A_1$. The alcoves in  $X$ are the the mirror images of $A_1$ under $W_p$ via the
dot action (i.e., the action with origin $-\rho$). 

Recall that as a consequence of the strong linkage principle \cite{A80a} all composition factors of a cohomology module $H^i(\lambda)$ have 
(dominant) 
highest weights
in $W_p \cdot \lambda$. Moreover, if for $\lambda, \mu \in \bar A_1$  we denote by $T_\lambda^\mu$ the translation functor from the $\lambda$-block to the $\mu$-block then we have
\begin{equation}
T_\lambda^\mu H^i(w \cdot \lambda ) = H^i(w \cdot \mu) \text { for all } \lambda \in F, \mu \in \bar F.
\end{equation}
Here $F$ denotes a non-0-facette of $\bar A_1$, i.e., $F$ is $A_1$ itself  or one of its walls. By $\bar F$ we denote the closure of $F$. 
We note that any $\lambda\in
X$ belonging to a 0-facette must be 
an element of
$-\rho+pX$ by our assumption
$p>5$.

By (2.1)
we immediately deduce that if a cohomology module vanishes for some weight in 
a facette (in our case an alcove or a wall) then the same cohomology
vanishes at any other weight in the closure of that facette. More generally, this translation principle allows us to determine the composition factors of the cohomology of a weight
in the closure of a facette once we have the corresponding information for a single weight in 
the facette. Therefore we abuse notation and sometimes write $H^i(F)$ for
a cohomology group of a (non-specified) weight in the facette $F$.  

We enumerate the first few alcoves in $X^+$ as in \cite{H}, see Fig. 1. The alcove containing the number $i$ is then denoted $A_i$. If $A_i$ and $A_j$ share a wall we denote this common wall
by $F_{i/j}$.  In Fig. 1 we have also given names to some of the Weyl chambers that we shall need often. We have set $x = st, y = ts, z = sy, \text { and } w= tx$.


If $v$ is any element of 
$W$
then we set $A_i^v = v \cdot A_i$ and similarly $F^v_{i/j} = v \cdot F_{i/j}$.
By the chamber
$v$ we will mean the chamber
$v\cdot X^+$.

Recall that $H^0(\lambda) \neq 0$ if and only if $\lambda \in X^+$. Moreover, Kempf's vanishing theorem says that $H^i(\lambda) = 0$ for all $i>0$ when $\lambda \in X^+$. Serre
duality implies that $H^i(\lambda)^* \simeq H^{6-i}(-\lambda -2\rho)$ for all $\lambda \in X$. 

In characteristic zero the vanishing behavior of line bundle cohomology (on any flag variety) is determined by Bott's theorem \cite{B} which says that $H^i(\lambda) \neq 0$ if and
only if $\lambda \in w \cdot X^+$ for some $w \in W$ of length $i$. By semi-continuity (or use the universal coefficient theorem in 
(3.2)
below) we have that the dimension of
a cohomology 
module
in characteristic $p>0$ is at least equal to the dimension of the corresponding cohomology module in characteristic $0$. Hence we have $H^i(\lambda) \neq 0$ in
all chambers $w \cdot X^+$ where $w \in W$ 
has
length $i$. In the following sections we will therefore concentrate to the ``extra" cohomology, i.e., on $H^i(\lambda)$'s or
$H^i(A)$'s for $\lambda$ or $A$ lying outside these chambers.

In \cite{A79} the first author obtained the complete vanishing description for $H^1(\lambda)$ and (via Serre duality) $H^{|R^+|-1}(\lambda)$, 
see also the alternative formulations in Corollary 2.7 and Remark 2.8 of \cite{A07}. When $\lambda$ belongs to the lowest $p^2$-alcoves we have in Fig.2 illustrated this result in our case
by entering $1$ (resp. $5$) on the alcoves
where extra $H^1$- (resp. $H^5$-)
occur.


The above remarks mean that the cohomology modules we have left to describe are $H^i(\lambda)$ with $2 \leq i \leq 4$. We deal with these after first describing the methods we use in the
next section.

\section{Methods}
In this section we will describe the methods we use to determine the vanishing or non-vanishing of a cohomology module $H^i(\lambda)$. The results may all be found in the
literature and are slight variations of the 6 step program sketched in \cite[Appendix B]{AK10}. They all apply in the case of a general flag variety although we have not always
chosen to formulate them in their most general versions. Our aim has been to gear them towards the $G_2$-problem of the present paper.

We set for each $n \geq 1$ 
\[
P_n = \{\lambda \in X \mid \langle \lambda + \rho, \gamma^{\vee}\rangle < p^n \text { for all } \gamma \in R\}.
\]
This is the union over all Weyl chambers of their lowest $p^n$-alcoves.
When $\lambda \in P_1$ the cohomology $H^\bullet(\lambda)$ behaves according to Bott's theorem in characteristic $0$, see \cite{A80a}. We shall use this to get our first results
for weights in $P_2$.

We call a weight $\nu \in X$ {\em special} if there exists $\lambda \in X$ such that $\nu = p\cdot \lambda$. Here we have extended the ``dot notation" to include also multiplication by
$p$ on $X$, i.e., $p \cdot \lambda = p(\lambda + \rho) - \rho
$.
Thus $\nu$ is special iff
$\nu\in
p\cdot(-\rho)+pX$. 
Note that $\lambda \in P_1$ if and only if $p\cdot \lambda \in P_2$. 
Also we say $\lambda$ is singular iff $\langle
\lambda + \rho, \gamma^{\vee} \rangle = 0$ for some $\gamma \in R$,
in which case $p \cdot \lambda$ is singular and vice versa. 
Otherwise we say
$\lambda$ is regular.
Likewise $p\cdot$ preserves the individual Weyl chambers.

Setting $X_p = \{\lambda \in X^+ \mid \langle \lambda, \gamma^{\vee}\rangle <p \text { for all simple roots } \gamma\}$ 
we then have the following proposition, cf. Step 1 in \cite[Appendix B]{AK10}

\begin{prop}
Let $\nu$ be a special point. Then 
\begin{itemize}
\item[i)] If $\nu \in \overline{P_2}$ 
and is singular then
$H^i(\nu) = 0$ for all $i\in \N$.

\item[ii)] If $\nu$ is regular and  $w$ is the element in $W$ for which $w\cdot \nu
\in X^+$ then $H^{l(w)}(\nu \pm \lambda) \neq 0$ for all $\lambda \in X_p$.

\end{itemize}
\end{prop}

\begin{proof}
i) is a consequence of the isomorphism $H^i(p\cdot \lambda) \simeq H^i(\lambda)^{(1)} \otimes St_p$, see \cite[Theorem 2.5]{A80b} (valid for all $\lambda \in X,  i\in \N$). 
Here $St_p =
H^0((p-1)\rho)$ is the Steinberg module. Moreover, ii) can also be deduced from this result, see \cite[Corollary 2.7]{A07}.
\end{proof}

\noindent
{\bf Remark 3.1.}
More generally,
if $\lambda\in X$ with
$H^i(\lambda)\ne0$, then
$\forall\mu\in p^r\cdot\lambda\pm X_{p^r}$ with
$r\in\N$,
$H^i(\mu)\ne0$.


\begin{proof}
If $\mu\in p^r\cdot\lambda- X_{p^r}$, this is \cite[Prop.2.6]{A07}.
If $\mu\in p^r\cdot\lambda+ X_{p^r}$, we argue as in
\cite[2.7]{A07}:
by the Serre-duality
$H^{N-i}(-\lambda-2\rho)\ne0$.
Then
$\forall\eta\in X_{p^r}$,
$0\ne
H^{N-i}(p^r\cdot(-\lambda-2\rho)-\eta)$.
It follows from the Serre-duality
again that
\begin{align*}
0
&\ne
H^{i}(-(p^r\cdot(-\lambda-2\rho)-\eta)-2\rho)
=
H^{i}(p^r\cdot\lambda+\eta).
\end{align*}
\end{proof}

We shall also need the following proposition obtained by the translation functors mentioned in Section2 . 
Let
$A$ be an alcove, $s$ the reflection in one of the walls of $A$  and denote by
$\theta_s$ the corresponding wall crossing functor. We will denote by $As$ the alcove adjacent to $A$ over the $s$-wall.
If $As > A$ then we have
the long exact sequence
\cite[Proposition 2.1]{A81}
\begin{equation}
\cdots \to H^i(A) \rightarrow \theta_s H^i(A) \rightarrow H^i(As) \rightarrow \cdots.
\end{equation}
Then part i) of the following  is 
(2.1).

\begin{prop}
\begin{itemize}
\item[i)] Let $F \subset X$ be a facette.
{If } $H^i(\lambda) = 0$  { for some } $\lambda \in F $ { then } 
$H^i(\mu) = 0$ for all $\mu \in \bar F$.
\item[ii)] Let $s$ be the reflection in a wall of an alcove $A$ and suppose $As > A$. Assume $H^{i+1}(A) = 0$.
\begin{itemize}
\item[a)] If $H^i(A) = 0$ then also $H^i(As) = 0$;
\item[b)] If $F$ denotes the common wall of $A$ and $As$ then $H^i(F)\neq 0$ if and only if $H^i(A) \neq 0 \neq H^i(As)$.
\end{itemize}
\item[iii)]
Suppose for some alcove $A$ we have $H^{i+1}(A')= 0$ for all alcoves which 
are obtained from $A$ by a sequence
of alcoves $A=A_1 < A_2< \cdots< A_r = A'$ with $A_{i+1} = A_is_i$ for some reflection $s_i$ in a wall of $A_i$, $i= 1, 2, \cdots , r-1$. If $H^i(A) = 0$ then also $H^i(A')= 0$ 
for all such $A'$.

\end{itemize}
\end{prop}

\begin{proof}
To prove ii) we observe that our assumption here makes the last displayed map  in (3.1) a surjection. This gives immediately part a). To obtain part b) we observe 
$\theta_s = T_\mu^\lambda \circ T_\lambda^\mu$ with $\lambda \in A$ and $\mu \in F$.
We have 
$T_\lambda^\mu H^i(A) \simeq H^i(F)
\simeq
T_\lambda^\mu H^i(As)$ 
which 
gives the ``only if" part. The ``if" part follows from the
surjection $\theta_s H^i(A) \simeq T_\mu^\lambda H^i(F) \to H^i(As)$. Finally iii) is obtained by a simple iteration of iia).
\end{proof}

Exploring the translation functors further one obtains (cf. Step 3 in \cite[Appendix B]{AK10}

\begin{prop} Let $\nu \in X$ be a special point and denote by $W_\nu$ the stabilizer of $\nu$ in $W_p$. Assume $H^j(\nu) = 0$ for some $j \in \N$.
\begin{itemize}
\item[i)]
If
$H^{j+1}(A) = 0$ for all alcoves $A$ with $\nu \in \overline{A}$,
then $H^j(\nu+\rho+A_1) = 0$.

\item[ii)] If in addition to the assumptions in i) we have $H^{j+1}(\nu + \lambda)= 0$ for all $\lambda \in X^+$ 
then also $H^j(\nu + \lambda) = 0$ for all $\lambda \in X^+$. 
\item[iii)]
Let $F$ be a facette with $\nu \in \bar F$. If $H^{j+1}(w \cdot F) = 0$ for all $w \in W_\nu$ and $F^+$ is maximal in $\{w\cdot F\}_{w \in W_\nu}$, then $H^j(F^+) = 0$.
\end{itemize}
\end{prop}

\begin{proof}
Note that
$W_\nu\simeq
W$.
Consider first iii).
By translating from the $\nu$-block to the $F$-block we obtain $H^j(F^+) = 0$, 
see 
\cite[Theorem
4.3]{A81}. Then i) is the special case where $F$ is an alcove. Finally 
ii) follows from i) and Proposition 3.2.iii).
\end{proof}
 Let $\Z_p$ be $\Z$ localized at the prime $p$ and denote for $\lambda \in X$ and $i\in\N$ by
$H_{\Z_p}^i(\lambda)$ the analogue of $H^i(\lambda)$ over $\Z_p$. Then by ``the universal coefficients theorem"
we get the short exact sequence
\begin{equation}
0 \to H_{\Z_p}^i(\lambda) \otimes k \to H^i(\lambda) \to
\Tor_1^{\Z_p}(H^{i+1}_{\Z_p}(\lambda), k) \to 0.
\end{equation}
This gives us the following result

\begin{prop}
Suppose $\lambda \in X$ does not belong 
to
a chamber $w\cdot X^+$ where $w \in W$ has
length $i+1$. If
$H^{i+2}(\lambda) = 0$ and $H^{i+1}(\lambda) \neq 0$ then also $H^i(\lambda) \neq 0$.
\end{prop}

\begin{proof} This is Step 4 ii) in \cite[Appendix B]{AK10} which also contains 
the easy proof.
\end{proof}

The following proposition 
is
a combination of Step 5 and 6 in \cite[Appendix B]{AK10}. 
\begin{prop}
Let
$\gamma$ be an arbitrary simple root with the associated reflection
$s_\gamma$,
and let
$\lambda \in X$  
also
be 
arbitrary, i.e., not necessarily in $P_2$.
\begin{itemize}
\item[i)]
If  $0 \leq \langle \lambda + \rho, \gamma^\vee \rangle=sp^m$
for some
$s, m\in\N$
with
$s<p$,
then
$H^{i+1}(s_\gamma\cdot~\lambda) \simeq H^i(\lambda)$.

\item[ii)]
If  $ap < \langle \lambda + \rho, \gamma^\vee \rangle < (a+1)p$  
for some 
$0<a<p$ then we have the following two long exact sequences 
\[
\cdots \to H^{i+1}(s_\gamma \cdot \lambda) \to H^i(\lambda) \to H^{i+1}(V) \to
\cdots 
\]
and
\[
\cdots \to H^{i+1}(C) \to H^{i+1}(V) \to H^i(Q) \to \cdots 
\] 
where $C$ and $Q $ both have weights
$s_\gamma\cdot\lambda +p\gamma,
s_\gamma\cdot\lambda +2p\gamma, \cdots , s_\gamma\cdot ~\lambda +ap\gamma$, each occurring with multiplicity $1$.
\item[iii)]
Suppose $\lambda$ satisfies the inequalities in ii). If $H^{i+1}(s_\gamma \cdot \lambda)$ has a composition factor $L(\mu)$ which is not a composition 
factor of any of the modules $H^j(s_\gamma \cdot \lambda + rp\gamma)$ for $j \in \{i, i-1\}$ and $1 \leq r \leq a$ then $L(\mu)$ must 
be a composition factor of $H^i(\lambda)$ forcing this module to be non-zero.

\end{itemize}
\end{prop}

\begin{proof}
i) and ii) are parts of the ingredients in the proof of the strong linkage principle \cite{A80a}. iii) is an immediate consequence of ii).
\end{proof}

\section{$H^4$}

\subsection{$H^4$ in the chambers $t, y$ and $w$}

First we apply Proposition 3.1. For each special point $\nu$ in the chamber $tz$ we have $H^4(\nu + \lambda) \neq 0$ for all $\lambda \in X_p$. This gives non-vanishing $H^4$ in the alcoves marked
$4$ in Fig. 3. Proposition 3.2.iib) then implies that $H^4$ is also non-zero on the 
walls
between these alcoves 
(we see from Figure 2 that the necessary $H^5$-vanishing needed to apply Proposition 3.2.iib) is indeed satisfied).
By Proposition 3.2.iii) we get also non-vanishing $H^4$ in the alcoves marked $\circ$,
and via Proposition 3.2.iib) also on the 
walls between them.

By Proposition 3.1.i) we have vanishing cohomology at all special points on the wall between chamber $w$ and chamber $t
z$. Proposition 3.3.i) then gives vanishing $H^4$ in the
alcoves marked $\times$ in
Fig. 3 (for alcove $A_{11}^w$ we apply 
Proposition 3.3.iii) with $F$ equal
 to its longest wall and then apply Proposition 3.2.ii). For the
 other alcoves we 
 apply Proposition 3.2.i)  to the alcove itself). Note that the short wall does not have a special point in its closure!

Now  Proposition 3.2.iii) allows us to deduce $H^4$-vanishing in all alcoves above the alcoves marked $\times$ (we say that an alcove $A'$ is above the alcove
$A$ iff there is a sequence of adjacent alcoves
$A<As_1<As_1s_2<\dots<As_1s_2\dots
s_r=A'$ as in Proposition 3.2.iii)).

Proposition 3.4 gives $H^4(A_4^w) \neq 0$ because $H^5(A^w_{4}) \neq 0$. 
The sequence 
(3.1)
gives us an exact sequence
\[ H^4(A_3^w) \to H^5(A_4^w) \to \theta_sH^5(A_4^w). \]
Here the last term vanishes because $H^5$ vanishes on the mid-sized wall (the $s$-wall) of $A_4^w$. On the other hand, the middle term is non-zero so that $H^4(A^w_3) \neq 0$. 
An analogous argument gives $H^4(A_5^w) \neq 0$ and then by Proposition 3.2.iia) also $H^4(A_6^w) \neq 0$.

If
$A=A_3^{tz}$
and $A'=A_6^{tz}$,
then
$H^4(F_{A^w_3/A})\ne0\ne
H^4(F_{A^w_6/A'})$ by Proposition 3.2.iib).

Consider now $F_{6/8}^w$. The sequences in Proposition 3.5.ii) are in this case
\[ 
H^4(F_{6/8}^w) \to H^5(V) \to H^6(F_{6/8}^{s
w}) \]
and
\[H^5 (C) \to H^5(V) \to H^4(Q) \to H^6(C). \]
The last term in the first sequence vanishes. In the second sequence the weights of $Q$ and $C$ are $F_{6/8}^{sw} + p\alpha$ and $ F_{6/8}^{sw} + 2p\alpha$. We have 
$H^\bullet(F_{6/8}^{sw} + p\alpha) = 0$ by Proposition 3.5.i) whereas $H^5(F_{6/8}^{s w} + 2p\alpha) = 0=H^6(F_{6/8}^{s w} + 2p\alpha)$ by the results in Section 2. 
It follows that $H^5(V) \simeq H^4(Q) \simeq H^4(F_{6/8}^{s w} + 2p\alpha) \neq
0$ and hence $H^4(F_{6/8}^w) \neq 0$. Then by Proposition 3.2.i) we have also $H^4(A_8^w) \neq 0$.

Applying the same method we get $H^4(F_{5/7}^{w}) = 0$. The difference here is that $F_{5/7}^{s w} +2p\alpha = F_{3/4}^w$ has vanishing $H^4$. In fact, weights on the wall $F_{3/4}$ are
minimal in $X^+$ (with respect to the strong linkage relation) so that  $H^\bullet(F^v_{3/4})$,
$v\in W$, behaves as in characteristic $0$, see \cite{A80a}. Moreover, $H^5(F^{s
w}_{5/7}) = 0$ and the sequences corresponding
to the above imply the claimed vanishing of  $H^4(F_{5/7}^w)$. Hence by Proposition 3.2.iib) we get $H^4(A_7^w) = 0$.

If $A=A_4^{tz}$, there are exact sequences
\[
H^5(F^{s}_{A_4^w/A})\to
H^4(F_{A_4^w/A})\to
H^5(V)\to
H^6(F^{s}_{A_4^w/A})
\]
and
\[
H^3(Q)\to
H^5(C)\to
H^5(V)\to
H^4(Q)
\]
with
$C=Q=F^{s}_{A_4^w/A}+p\alpha$.
As $F^{s}_{A_4^w/A}+p\alpha\in\overline{A_1}$,
$H^i(F^{s}_{A_4^w/A}+p\alpha)\ne0$
iff $i=5$.
As
$H^5(F_{A_4^w/A})=0=
H^6(F_{A_4^w/A})$,
it follows that
$H^4(F_{A_4^w/A})\simeq
H^5(V)\simeq
H^5(C)=
H^5(F_{A_4^w/A}^{s}+p\alpha)
\ne0$.

As $A_1$ is the bottom dominant alcove, by the strong linkage principle
again
$H^4(A^w_1) = 0$.

In Fig. 3 we have 
marked $\lozenge$
on the alcoves $A_i^w$, $i = 3, 4, 5, 6$
to indicate that they have non-vanishing $H^4$ whereas we have 
marked
$\blacklozenge$
on $A_i^w$ 
for $i = 1, 
7$ because
here we have vanishing $H^4$. By Proposition 3.2.iii)
we then have vanishing $H^4$ also for all alcoves above either $A_1^w$ or $A_7^w$.

This completes the description of  the vanishing behavior of $H^4$ in the intersection of $P_2$ with the chambers $t, y$ and $w$,
and on the facettes between them by Proposition 3.2.iib).
In particular, we record ``exceptional" vanishing
$H^4(F^w_{3/4})=0$
in contrast to a statement in \cite[Section 5]{A81}, cf. also \cite{H},
by marking
$\times$ on the wall.

 \subsection{$H^4$ in the chambers $s, x$ and $z$}
 We proceed exactly as above and record the results in Fig. 4. First we get the non-vanishing of $H^4$ in the
 alcoves marked $4$
 and then in the 
 alcoves marked with $\circ$.
 By Proposition 3.2.iib) $H^4$ is also non-zero on the facettes between these alcoves.
 Next we locate those special points on the wall between chamber $z$ and chamber $s
 w$ at which $H^4$ vanishes and where moreover the method in Proposition 3.3 ensures vanishing in the cones above
 them. We indicate this by marking
 $\times$ on the bottom alcoves in these cones. 
 As $H^5(A^z_i)\ne0$,
 $i = 10, 11, 12$,
 Proposition 3.4 gives us that $H^4$ is non-zero 
 in these alcoves. 
Then Proposition 3.2.ii) ensures that $H^4$ is likewise non-zero
 in the adjacent alcoves corresponding to $i = 13, 14, 8$ and $9$.
Non-vanishing in $A^z_8$ in turn implies for the same reason non-vanishing also in $A^z_7$ and $A^z_9$. We have 
 marked all these
 alcoves 
 with $\lozenge$
 in Fig. 4.


 We now observe that by Proposition 3.5 we have $H^j(F^z_{11/13}) \simeq H^{j+1}(F^{t
 z}_{11/13}) = 0$ for $j \geq 4$. The arguments based on Proposition 3.2 then give $H^4(A_{13}^z) \simeq H^5(A^z_{11})$
 so that by translating to the wall we get $H^4(F^z_{13/15}) \simeq H^5(F^z_{8/11}) = 0$. It follows via Proposition 3.2.iib) that $H^4$ vanishes on $A_{15}^z$ (which we therefore 
 mark with $\blacklozenge$) as well as on
 all alcoves above 
 it
 (Proposition 3.2.iii)). 
 
Consider now the exact sequences coming from Proposition 3.5.ii)
\[
H^4(V)\to
H^5(F_{11/8}^{t z})\to
H^4(F_{11/8}^z)\to
H^5(V)
\]
and
\[
H^4(C)\to
H^4(V)\to
H^3(Q)\to
H^5(C)\to
H^5(V)\to
H^4(Q)
\]
with
$C=Q=F^{t z}_{11/8}+p\beta$.
As $F^{t z}_{11/8}+p\beta$ is a $W$-conjugate of the long wall of $A_1$
we have
$H^\bullet(F^{t z}_{11/8}+p\beta)=0$.
It follows that
$H^4(V)=0=H^5(V)$, and hence
$H^4(F_{11/8}^z)\simeq
H^5(F_{11/8}^{t z})=0$.
Then $H^4(A_8^z)\simeq
H^5(A_{11}^z)$
by Proposition 3.2.
Translating to the wall
we obtain
$H^4(F_{8/6}^z)\simeq
H^5(F_{13/15}^z)=0$, and hence
$H^4(A_6^z)=0$ again by by Proposition  3.2.
Hence we 
mark $A_6^z$ with
$\blacklozenge$ and also we get vanishing in all alcoves above it.  
We see likewise 
$H^4(F^z_{14/12})\simeq
H^5(F^{t z}_{14/12})=0$
while
$H^4(F^z_{10/9})\simeq
H^5(F^{t z}_{10/9})\ne0$.

The sequences for $F_{5/7}^z$ analogous to the ones above give $H^4(F_{5/7}^z) \simeq H^5(F_{5/7}^{tz}) \neq 0$ 
(we use for this the observation that $H^\bullet(F^{tz}_{5/7} + p\beta) =0$). Hence also $H^4(A_5^z) \neq 0$ and we
have therefore equipped this alcove with a $\lozenge$ in Fig. 4.

 When $i = 1, 2, 3$ we have (Proposition 3.5.i)) $H^4(A^z_i) \simeq H^5(A^{t
 z}_i) = 0$ giving rise to  
 $\blacklozenge$ marking of these alcoves. 
 

Finally, we examine $H^4(F^z_{10/12}),
H^4(F^z_{11/12})$,
and
$H^4(F_{A_{10}^z/A_{10}^{t w}})$.
Put for simplicity $
F=F_{A_{10}^z/A_{10}^{t w}}$.
There are exact sequences
\[
H^4(V)\to
H^5(t\cdot
F)\to
H^4(
F)
\]
and
\[
H^4(C)\to
H^4(V)\to
H^3(Q)
\]
with
$C=Q=t\cdot
F+p\beta$.
Another application of  Proposition 3.5.i)
gives
$H^i(t\cdot
F+p\beta)\simeq
H^{i+1}(t\cdot((t\cdot
F+p\beta))=0$
unless $i=5$.
It follows that
$H^4(V)=0$,
and hence that
$H^4(
F)$ surjects onto
$H^5(t\cdot
F)\ne0$.
Likewise 
$H^4(F^z_{10/12})\simeq
H^5(F^{t z}_{10/12})=0=
H^5(F^{t z}_{11/12})\simeq
H^4(F^z_{11/12})$.

 This
 completes the the description of the vanishing behavior of $H^4$ in 
 the intersection of $P_2$ with the chambers $z, x$ and $s$,
 and also on the facettes between them.
 We record
$H^4(F^z_{11/13})=
H^4(F^z_{14/12})=
H^4(F^z_{10/12})=
H^4(F^z_{11/12})=0$, see again \cite{A81} and \cite{H}.

\section{$H^3$}

\subsection{$H^3$ in the chambers $t$ and $y$}
 The same methods as used for $H^4$ above give non-vanishing of $H^3$ in the alcoves marked $3$ and $\circ$, and vanishing in the alcoves
 marked $\times$ or $\blacklozenge$
 and all alcoves above such,
 see Fig. 5. Moreover, we also have $H^3(A_1^y) = 0$ so that we can mark this alcove by a $\blacklozenge$. As before we have vanishing $H^3$ in alcoves above it. 
 
 Our $H^4$-results in Section 4 give non-vanishing $H^4$ in the alcoves $A^y_i$ for $i = 20, 22, 24$. By Proposition 3.4 this implies non-vanishing of  $H^3$ in the same alcoves and then by the exact sequence (3.1) also
 first in the alcove $A_{18}^y$ 
 as
 $H^4(F^y_{18/22})=0$,
 and then 
 by Proposition 3.2.iia)
 in $A_i^y$, $i = 9, 10, 12, 14, 17$. We have therefore marked all these alcoves $\lozenge$. 
Proposition 3.4 
gives non-vanishing 
of
$H^3(F^y_{22/20}),
H^3(F^y_{24/20})$
and
$H^3(F_{A^y_{24}/A^w_{24}})$
among the walls of the alcoves   
$A^y_i$ for $i = 20, 22, 24$.
Moreover,
\begin{align*}
H^3(F^y_{22/25})
&\simeq
H^4(F^z_{22/25})
\quad\text{by Proposition 3.5.i)}
\\
&\ne0
\quad\text{as we have seen in \S4},
\\
H^3(F^y_{22/18})
&\simeq
H^2(F^s_{22/18})
\quad\text{by Proposition 3.5.i)}
\\
&=0
\quad\text{as we will see in 6.2},
\\
H^3(F^y_{20/17})
&\simeq
H^2(F^s_{20/17})
\quad\text{by Proposition 3.5.i)}
\\
&=0
\quad\text{as we will see also 6.2},
\\
H^3(F^y_{24/31})
&\simeq
H^4(F^z_{24/31})
\quad\text{by Proposition 3.5.i)}
\\
&\ne0
\quad\text{as we
have seen in \S4}.
\end{align*}
The remaining walls in this chamber are now covered by Proposition 3.2 iib).


 Consider then $A^y_7$. By Proposition 3.5.ii) we have the exact sequences
 \[ H^1(A^{s}_7) \to H^2(V) \to H^3(A_7^y) \]
 and
 \[ H^2(V) \to H^1(A_5^{s}) \to H^3(A_5^{s}). \]
By Proposition 3.5.i),
in the first sequence $H^1(A^{s}_7) \simeq H^0(A_7)$ and in the second sequence  $H^1(A_5^{s}) \simeq H^0(A_5)$ whereas $H^3(A_5^{s})\simeq
H^2(A_5)=0$. Now by \cite{A86}
$L(A_4)$ is a composition factor of $H^0(A_5)$ but not of $H^0(A_7)$. We conclude that $H^3(A^y_7) \neq 0$. 


Likewise we have exact sequences
 \[ 
 H^2(V) \to H^3(A_8^y)\to
  H^2(A_8^s) \]
  and
 \[ 
   H^2(A_6^s)\to
   H^2(V) \to H^1(A_6^{s}) \to H^3(A_6^{s}) \]
with
$H^j(A_6^s)\simeq
H^{j-1}(A_6)$
by Proposition 3.5.i), which vanishes unless $j=1$.
As $H^2(A_8^s)\simeq H^1(A_8)=0$
by Proposition 3.5.i)
again,
we get
$H^0(A_6)\twoheadrightarrow
H^3(A_8^y)$.
But the composition factors of $H^0(A_6)$ are $L(A_6)$ and $L(A_5)$
while $H^3(F_{8/6}^y)=0$
as
$H^3(A_6^y)=0$.
As the translation to the $F_{8/6}$-wall annihilates
neither 
$L(A_6)$ nor $L(A_5)$,
we must have
$H^3(A_8^y)=0$
(which we therefore marked $\blacklozenge$) and on all alcoves above it.

This completes the determination of $H^3$ in these chambers.
We record
$H^3(F^y_{22/18})
=
H^3(F^y_{20/17})=0$.

\subsection{$H^3$ in the chambers $s$ and $x$}

Once again we get non-vanishing in the alcoves marked $3$ and $\circ$ as well as vanishing in the alcoves marked $\times$ and all alcoves above them.

By Proposition 3.5.i) we get $H^3(A^x_1) \simeq H^2(A_1) = 0$,  
so we have marked 
alcove
$A_1^x$ with $\blacklozenge$ and have vanishing of $H^3$ in all alcoves above 
it.
On the other hand we get $H^3(A^x_3) \simeq H^4(A_3^w) \neq 0$.

Weights on the facet $F_{3/4}$ are minimal in $X^+$ (with respect to the strong linkage). Hence $H^3(F^x_{3/4}) = 0$
by
\cite{A80a}.
This gives vanishing of $H^3$ first on $A^x_4$ (marked $\blacklozenge$ in Fig. 6)
by Proposition 3.2.iib) and
then on all alcoves above it by Proposition 3.2.iii).

The $\beta$-sequences from Proposition 3.5.ii) relative to $A^x_8$ give $H^3(A_8^x) \simeq H^4(A^w_8) \neq 0$. So we mark $A^x_8$ and subsequently $A^x_6$ by $\lozenge$
from the exact sequence (2.1)
as
$H^4(A^x_6)=0$. The same sequences with respect to
$A^x_{11}$ give $H^3(A^x_{11}) \simeq H^4(A_{11}^w) = 0$. We mark $A^x_{11}$
with $\blacklozenge$
 and 
 $H^3$ vanishes on all alcoves above it.

Note that $H^4(A^x_{23}) \neq 0$. By Proposition 3.4 this means that also $H^3(A_{23}^x) \neq 0$. The 
exact sequence (2.1)
then gives non-vanishing $H^3$ in the adjacent alcove
$A^x_{21}$ as well as in the alcove $A_{19}^x$ below it. We mark these alcoves $\lozenge$ in Fig. 6.

By Proposition 3.5.i) we have
$H^3(F_{23/26}^x)
\simeq
H^4(F_{23/26}^w)\ne0$.
Likewise
$H^3(F_{23/21}^x)
\simeq
H^2(F_{23/21}^t)$
which does not vanish as we will see in
6.1.
Also, as
$H^4(F_{A^x_{23}/A^z_{23}})\ne0$
by
Proposition 3.2.iib), we get from Proposition 3.4
that
$H^3(F_{A^x_{23}/A^z_{23}})\ne0$.

Consider now $F_{22/25}^x$. The exact $\beta$-sequences from Proposition 3.5 give
\[ H^4(F_{22/25}^w) \to H^3( F_{22/25}^x) \to H^4(V) \]
and
\[ H^4(C) \to H^4(V) \to H^3(Q).\]
 Here our $H^4$-results show that $H^4(F_{22/25}^w) = 0$ and by looking at the weights $F_{22/25}^w +p \beta$ and $F_{22/25}^w + 2 p \beta$ of both $C$ and $Q$ we see that $H^4(C)
 = 0 = H^3(Q)$. Hence $H^4(V) = 0$ and it follows that $H^3( F_{22/25}^x)= 0$. Therefore also $H^3(A_{25}^x) = 0$
 by Proposition 3.2.ii)
 as
 $H^4(A^x_{22})=0$. We mark this alcove $\blacklozenge$ in Fig. 6 and have vanishing $H^3$ in all
 alcoves above it.
 
 We claim that $H^3(A_{22}^x) \neq 0$. To see this consider the $\alpha$-sequences in Proposition 3.5.ii)
 \[ H^1(A_{22}^{t}) \to H^2(V) \to H^3(A_{22}^x)\to
 H^2(A^t_{22}) \]
 and \[ H^2(V) \to H^1(Q) \to H^3(C).\]
 The weights of $C$ and $Q$ are $\{A_{22}^x + jp\alpha \mid j = 1,2,3 \}$. 
We know
$H^3(A_{22}^x + p\alpha)=0$.
Also by Proposition 3.5.i) and
by Kempf
we have
$H^3(A_{22}^x + 3p\alpha)\simeq
H^2(t\cdot(A_{22}^x + 3p\alpha))=0=H^3(A_{22}^x + 2p\alpha)$;
\begin{align*}
H^3(A_{22}^x + 3p\alpha)
&=
H^3(A_{20}^t)
\simeq
H^2(A_{20})
\quad\text{by Proposition 3.5.i)}
\\
&=
0
\quad\text{by Kempf}
\\
&=
H^3(A_{13})
=H^3(A_{22}^x + 2p\alpha)
\quad\text{by Kempf again}.
\end{align*}
It follows that
$H^3(C)=0$, and hence we obtain an epi
$H^2(V)\twoheadrightarrow
H^1(Q)$.
Likewise, as
$H^2(A^t_{22}) \simeq
H^1(A_{22})=0$
and as
$H^1(A^t_{22}) \simeq
H^0(A_{22})$
both by Proposition 3.5.i),
we obtain an exact sequence
$H^0(A_{22})\to
H^2(V)\to
H^3(A_{22}^x)\to0$.
Moreover, 
we have an exact sequence
\[ H^0(Q') \to H^1(A_{22}^x + p\alpha) \to H^1(Q),\]
where the weights of $Q'$ are $A_{22}^x + 2p\alpha$ and $A_{22}^x + 3p\alpha$. As $A_{22}^x + 3p\alpha$ is non-dominant it follows that $H^0(Q') \simeq  H^0(A_{22}^x + 2p\alpha) =
H^0(A_{13})$. 
On the other hand,
the $\alpha$-sequences in Proposition 3.5.ii) yields
an exact sequence
\[
0\to
H^0(A_{22}^x + 2p\alpha)\to
H^1(A_{22}^x + p\alpha)\to
H^0(s\cdot(A_{22}^x + p\alpha))
\to
H^1(A_{22}^x + 2p\alpha)
\]
with
$H^0(A_{22}^x + 2p\alpha)=H^0(A_{13})$,
$H^1(A_{22}^x + 2p\alpha)=H^1(A_{13})=0$,
and $H^0(s\cdot(A_{22}^x + p\alpha))=H^0(A_{14})$.
Now we use the observation in Proposition 3.6 that 
$H^0(A_{14})$ has composition factor
$L(A_6)$
while 
$H^0(A_{13})$
does not.
It follows that
$H^1(A_{22}^x + p\alpha)$,
and hence also
$H^1(Q)$ and $H^2(V)$ have 
composition factor
$L(A_6)$.
As $H^0(A_{22})$
does not have composition factor
$L(A_6)$
again by
\cite{A86},
we conclude that
$L(A_6)$ is a composition factor of
$H^3(A_{22}^x)$ proving our claim.

If $p=7$, as $A_{13}, A_{14}$ and $A_{22}$ do not live in the Jantzen region, \cite{A86} does not directly apply.
One can, however, compute the multiplicity of
$L(A_6)$ in
$H^0(A_{13}),
H^0(A_{14}),
H^0(A_{22})$
using the $G_1T$-version of the Lusztig conjecture,  $G_1$ the Frobenius kernel of $G$,
which holds 
thanks to 
\cite{JCJ},
by recalling 
\[
[H^0(A) : L(C)]
=
\sum_{\nu\in X}
[\hat\nabla(A) : \hat L(\nu)][L(\nu^0)\otimes\chi(\nu^1)^{[1]} : L(C)],
\]
where $\nu=\nu^0+p\nu^1$
with
$\nu_0\in X_p$ and
$\nu^1\in X$,
$\chi(\nu^1)=\sum_i(-1)^iH^i(\nu^1)$, and
$[\hat\nabla(A) : \hat L(\nu)]$ is the multiplicity of 
the
$G_1T$-irreducible
$\hat L(\nu)$ of highest weight $\nu$
in the $G_1T$-composition series of
$\hat\nabla(A)=\ind_{B_1T}^{G_1T}(A)$
with $B_1$
the Frobenius kernel $B$.

 It follows from
 the exact sequence (2.1)
 that also the alcoves $A^x_{18}$,   $A^x_{16}$,   $A^x_{15}$, $A^x_{21}$, and
 $A^x_{19}$ have non-vanishing $H^3$ and they have been marked accordingly in Fig. 6.
 

 \section{$H^2$}
 
 \subsection{$H^2$ in 
 the
 chamber $t$}
 
Propositions 3.1-3 suffice to determine the vanishing behavior of $H^2$ in this chamber, see Fig. 7.

\begin{figure}
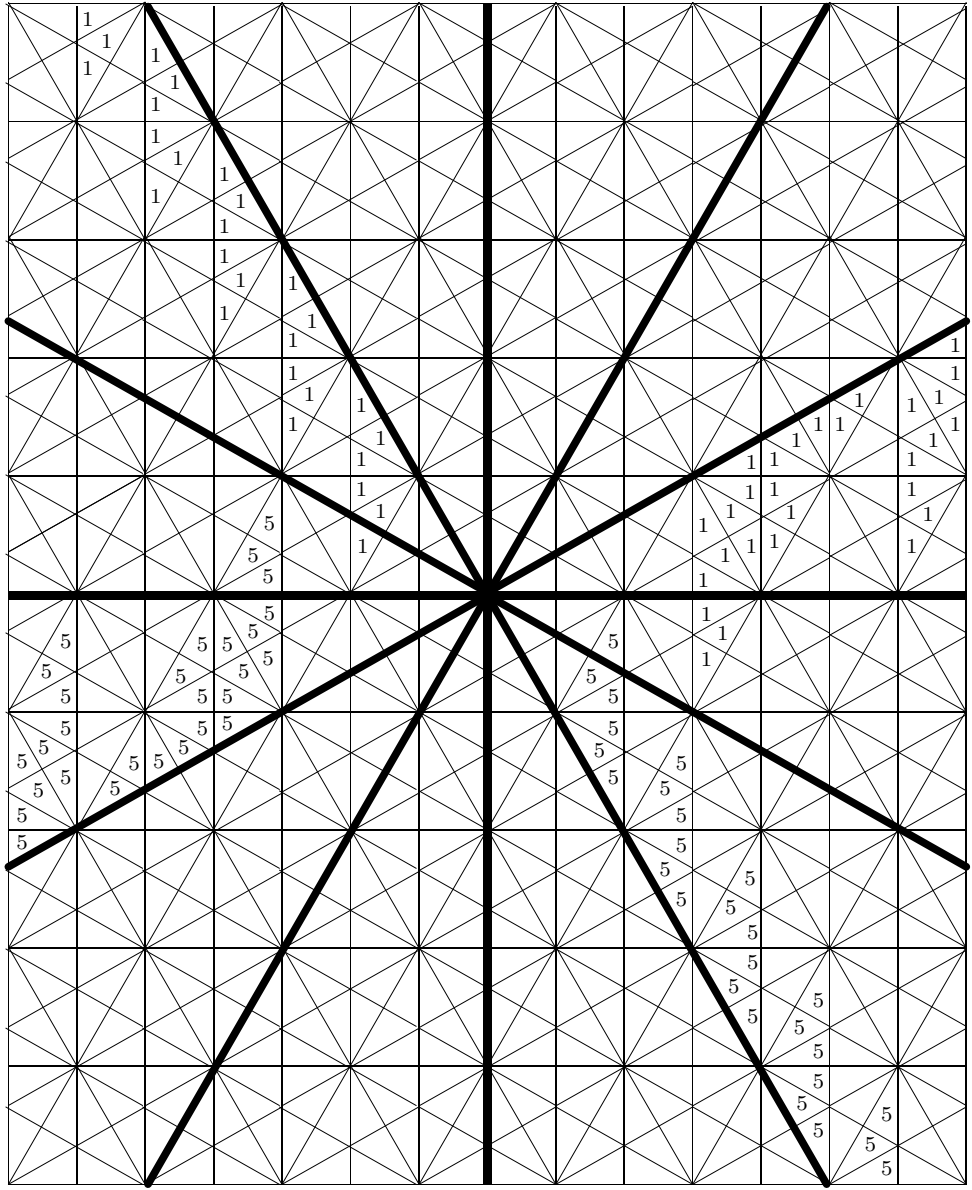

\begin{center}
{\tiny%
  \setlength{\unitlength}{3.6mm}
}

\caption{Extra $H^1$- and $H^5$-nonvanishing}
\end{center}
\end{figure}

\begin{figure}
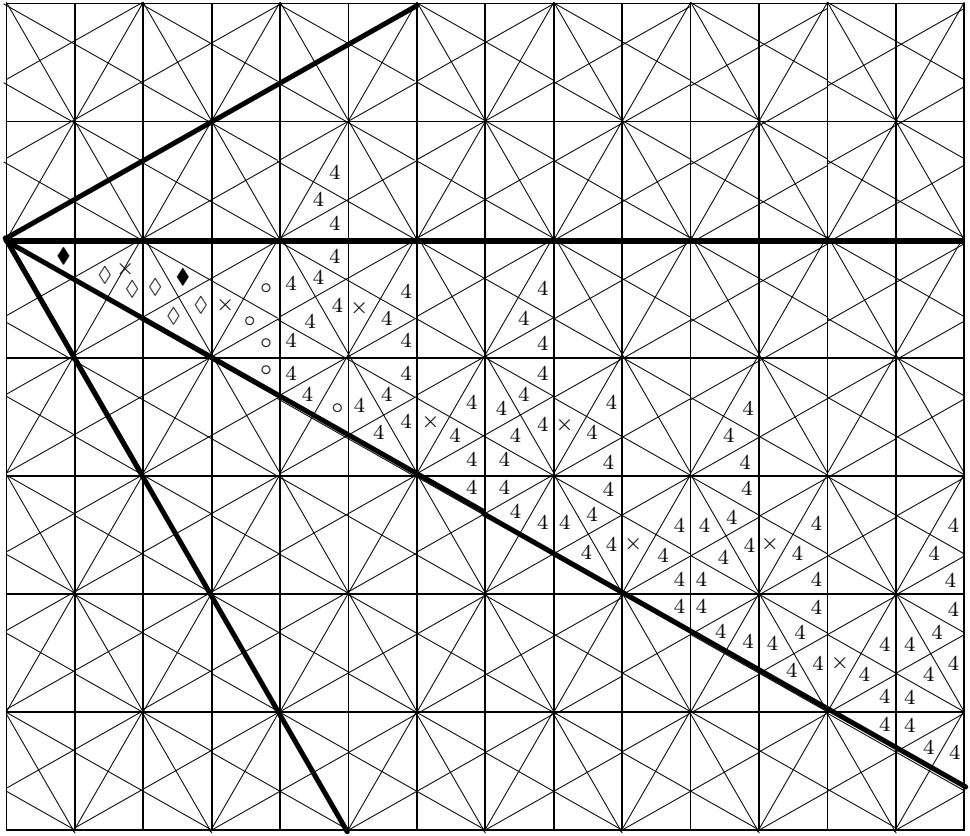

\begin{center}
{\tiny%
  \setlength{\unitlength}{3.6mm}
  %
}
\vspace{-5cm}
\caption{Extra $H^4$-nonvanishing in the chambers $t, y$ and $w$}
\end{center}
\end{figure}

\begin{figure}
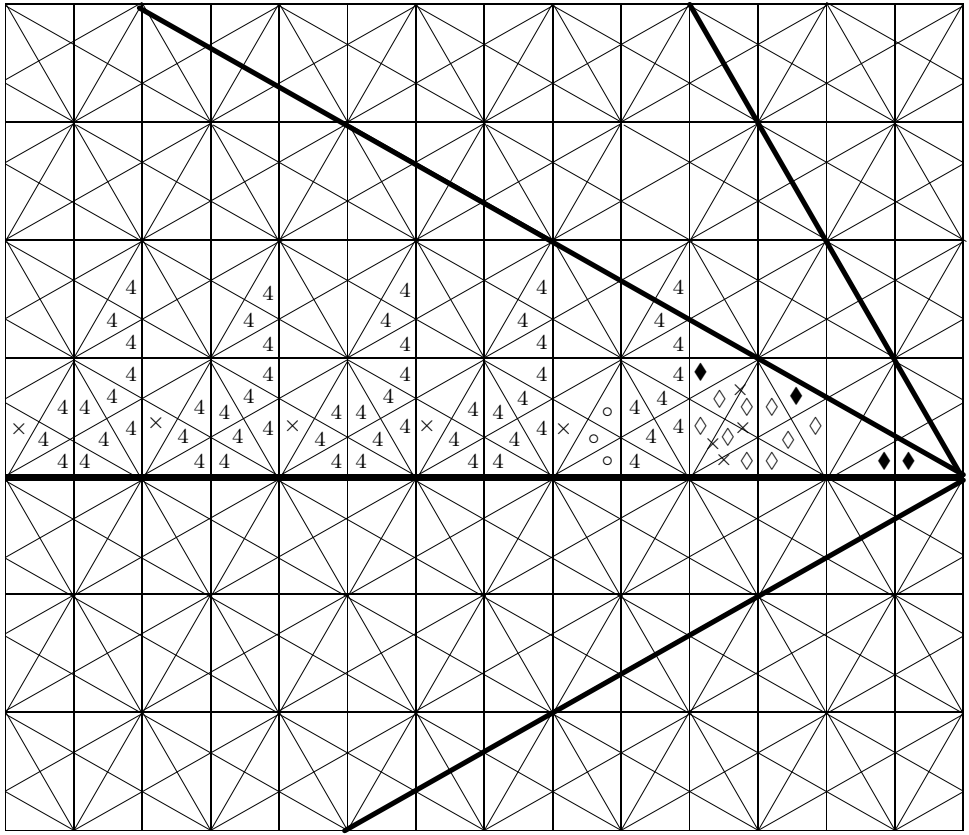

\begin{center}
{\tiny%
  \setlength{\unitlength}{3.6mm}
  %
}
\caption{Extra $H^4$-nonvanishing in the chambers $s, x$ and $z$}
\end{center}
\end{figure}

\begin{figure}
\begin{center}
{\tiny%
  \setlength{\unitlength}{3.6mm}
  %
}
\caption{Extra $H^3$-nonvanishing in the chambers $t$ and $y$}
\end{center}
\end{figure}

\begin{figure}
\begin{center}
{\tiny%
  \setlength{\unitlength}{3.6mm}
  %
}
\caption{Extra $H^3$-nonvanishing in the chambers $s$ and $x$}
\end{center}
\end{figure}

\begin{figure}
\begin{center}
{\tiny%
  \setlength{\unitlength}{3.6mm}
  %
}
\caption{Extra $H^2$-nonvanishing in the chamber $t$}
\end{center}
\end{figure}

 \subsection{$H^2$ in the chamber $s$}
Propositions 3.1-3 again settle most alcoves in this chamber and they give 
 $H^2(A_1^s)=0$. The only alcoves needing special considerations are $A^s_j$ with $j = 9, 10, 12, 14, 17, 18, 24, 31, 32, 33, 34, 35$.
 We claim that $H^2$ is non-zero in the first $10$ of those but zero on the last two. 
To prove this, using the exact sequence
 (2.1)
 and 
 Proposition 3.2.iia),  it is enough to 
 show
 that $H^2(A^s_{18}) \neq 0
 \neq H^2(A^s_{33})$ whereas $H^2(A_{34}^s) = 0$.
See Fig. 8.

\begin{figure}
\begin{center}
{\tiny%
  \setlength{\unitlength}{3.6mm}
  \begin{picture}(40,44)

\put(12.1,36.7){\makebox(0,0){2}}
\put(11.5,37.7){\makebox(0,0){2}}
\put(12.1,38.6){\makebox(0,0){2}}



\put(14.6,32.5){\makebox(0,0){2}}
\put(14,33.4){\makebox(0,0){2}}
\put(14.6,34.3){\makebox(0,0){2}}

\put(14.6,35.3){\makebox(0,0){2}}
\put(14.6,37.2){\makebox(0,0){2}}
\put(13.9,36.2){\makebox(0,0){2}}
\put(13,36.7){\makebox(0,0){$\times$}}


\put(17.1,28){\makebox(0,0){2}}
\put(17.1,29.8){\makebox(0,0){2}}
\put(16.5,29){\makebox(0,0){2}}
\put(15.5,32.4){\makebox(0,0){$\times$}}

\put(17.1,31){\makebox(0,0){2}}
\put(17.1,32.9){\makebox(0,0){2}}
\put(16.4,31.9){\makebox(0,0){2}}


\put(19.6,23.6){\makebox(0,0){2}}
\put(19.6,25.5){\makebox(0,0){2}}
\put(19,24.6){\makebox(0,0){2}}

\put(19.6,26.7){\makebox(0,0){2}}
\put(19.6,28.6){\makebox(0,0){2}}
\put(18.9,27.5){\makebox(0,0){2}}
\put(18,28){\makebox(0,0){$\times$}}


\put(22,21.3){\makebox(0,0){2}}
\put(22,19.3){\makebox(0,0){2}}
\put(21.5,20.3){\makebox(0,0){2}}

\put(20.4,23.5){\makebox(0,0){$\times$}}
\put(21.5,23.1){\makebox(0,0){2}}
\put(22,22.3){\makebox(0,0){2}}
\put(22,24.1){\makebox(0,0){2}}




\put(23,19.4){\makebox(0,0){$\times$}}
\put(24.6,14.9){\makebox(0,0){$\lozenge$}}
\put(24,15.9){\makebox(0,0){$\lozenge$}}
\put(24.6,16.8){\makebox(0,0){$\lozenge$}}

\put(23.9,18.8){\makebox(0,0){$\blacklozenge$}}
\put(24.6,17.9){\makebox(0,0){$\lozenge$}}






\put(27,12.4){\makebox(0,0){$\lozenge$}}
\put(26.4,11.5){\makebox(0,0){$\lozenge$}}
\put(27,10.6){\makebox(0,0){$\lozenge$}}

\put(25.5,14.9){\makebox(0,0){$\times$}}
\put(27,13.6){\makebox(0,0){$\lozenge$}}
\put(26.4,14.5){\makebox(0,0){$\lozenge$}}
\put(27,15.4){\makebox(0,0){$\lozenge$}}

\put(29.5,6.3){\makebox(0,0){$\blacklozenge$}}
\multiput(2.5,0)(5,0){7}{\line(4,7){2.5}}
\multiput(2.5,0)(5,0){7}{\line(-4,7){2.5}}
\multiput(7.5,0)(5,0){6}{\line(-7,4){7.5}}
\multiput(2.5,0)(5,0){1}{\line(-7,4){2.5}}
\multiput(2.5,0)(5,0){6}{\line(7,4){7.5}}
\multiput(32.5,0)(5,0){1}{\line(7,4){2.5}}
\multiput(0,1.45)(5,0){1}{\line(7,4){5}}
\multiput(35,1.45)(5,0){1}{\line(-7,4){5}}

\multiput(0,4.3)(5,0){7}{\line(4,7){2.5}}
\multiput(5,4.3)(5,0){7}{\line(-4,7){2.5}}
\multiput(10,4.3)(5,0){6}{\line(-7,4){7.5}}
\multiput(5,4.3)(5,0){1}{\line(-7,4){5}}
\multiput(0,4.4)(5,0){6}{\line(7,4){7.5}}
\multiput(30,4.3)(5,0){1}{\line(7,4){5}}
\multiput(0,7.25)(5,0){1}{\line(7,4){2.5}}
\multiput(35,7.25)(5,0){1}{\line(-7,4){2.5}}

\multiput(2.5,8.6)(5,0){7}{\line(4,7){2.5}}
\multiput(2.5,8.6)(5,0){7}{\line(-4,7){2.5}}
\multiput(7.5,8.6)(5,0){5}{\line(-7,4){7.5}}
\multiput(32.5,8.6)(5,0){1}{\line(-7,4){7.5}}\multiput(2.5,8.6)(5,0){1}{\line(-7,4){2.5}}
\multiput(2.5,8.65)(5,0){6}{\line(7,4){7.5}}
\multiput(32.5,8.65)(5,0){1}{\line(7,4){2.5}}
\multiput(0,10.05)(5,0){1}{\line(7,4){5}}
\multiput(35,10.15)(5,0){1}{\line(-7,4){5}}

\multiput(0,12.9)(5,0){7}{\line(4,7){2.5}}
\multiput(5,12.9)(5,0){7}{\line(-4,7){2.5}}
\multiput(10,13.05)(5,0){3}{\line(-7,4){7.5}}
\multiput(24.95,13.05)(5,0){2}{\line(-7,4){7.5}}\multiput(34.9,13.1)(5,0){1}{\line(-7,4){7.5}}
\multiput(5,12.9)(5,0){1}{\line(-7,4){5}}
\multiput(0,13.05)(5,0){6}{\line(7,4){7.5}}
\multiput(30,12.9)(5,0){1}{\line(7,4){5}}
\multiput(0,15.9)(5,0){1}{\line(7,4){2.5}}
\multiput(35,15.85)(5,0){1}{\line(-7,4){2.5}}

\multiput(2.5,17.4)(5,0){7}{\line(4,7){2.5}}
\multiput(2.5,17.4)(5,0){7}{\line(-4,7){2.5}}
\multiput(12.5,17.4)(5,0){1}{\line(-7,4){7.5}}
\multiput(17.5,17.4)(5,0){3}{\line(-7,4){7.5}}
\multiput(32.5,17.4)(5,0){1}{\line(-7,4){7.5}}
\multiput(2.5,17.4)(5,0){1}{\line(-7,4){2.5}}
\multiput(2.5,17.4)(5,0){6}{\line(7,4){7.5}}
\multiput(32.5,17.4)(5,0){1}{\line(7,4){2.5}}
\multiput(0,18.85)(5,0){1}{\line(7,4){5}}
\multiput(35,18.85)(5,0){1}{\line(-7,4){5}}

\multiput(0,21.7)(5,0){7}{\line(4,7){2.5}}
\multiput(5,21.7)(5,0){4}{\line(-4,7){2.5}}
\multiput(25,21.7)(5,0){3}{\line(-4,7){2.5}}\multiput(10.1,21.7)(5,0){2}{\line(-7,4){7.5}}
\multiput(20.1,21.7)(5,0){4}{\line(-7,4){7.5}}
\multiput(5,21.7)(5,0){1}{\line(-7,4){5}}
\multiput(0,21.75)(5,0){6}{\line(7,4){7.5}}
\multiput(30,21.7)(5,0){1}{\line(7,4){5}}
\multiput(0,24.6)(5,0){1}{\line(7,4){2.5}}
\multiput(35,24.6)(5,0){1}{\line(-7,4){2.5}}

\multiput(2.5,26.1)(5,0){7}{\line(4,7){2.5}}
\multiput(2.5,26.1)(5,0){7}{\line(-4,7){2.5}}
\multiput(7.5,26.1)(5,0){2}{\line(-7,4){7.5}}
\multiput(17.5,26.1)(5,0){4}{\line(-7,4){7.5}}
\multiput(2.5,26.1)(5,0){1}{\line(-7,4){2.5}}
\multiput(2.5,26.1)(5,0){6}{\line(7,4){7.5}}
\multiput(32.5,26.1)(5,0){1}{\line(7,4){2.5}}
\multiput(0,27.65)(5,0){1}{\line(7,4){5}}
\multiput(35,27.65)(5,0){1}{\line(-7,4){5}}

\multiput(0,30.4)(5,0){7}{\line(4,7){2.5}}
\multiput(5,30.4)(5,0){2}{\line(-4,7){2.5}}
\multiput(20,30.4)(5,0){4}{\line(-4,7){2.5}}\multiput(10.1,30.4)(5,0){6}{\line(-7,4){7.5}}
\multiput(5,30.4)(5,0){1}{\line(-7,4){5}}
\multiput(0,30.45)(5,0){6}{\line(7,4){7.5}}
\multiput(30,30.4)(5,0){1}{\line(7,4){5}}
\multiput(0,33.35)(5,0){1}{\line(7,4){2.5}}
\multiput(35,33.3)(5,0){1}{\line(-7,4){2.5}}

\multiput(2.5,34.8)(5,0){7}{\line(4,7){2.5}}
\multiput(2.5,34.8)(5,0){7}{\line(-4,7){2.5}}
\multiput(7.5,34.8)(5,0){6}{\line(-7,4){7.5}}
\multiput(2.5,34.8)(5,0){1}{\line(-7,4){2.5}}
\multiput(2.5,34.8)(5,0){6}{\line(7,4){7.5}}
\multiput(32.5,34.8)(5,0){1}{\line(7,4){2.5}}
\multiput(0,36.3)(5,0){1}{\line(7,4){5}}
\multiput(35,36.3)(5,0){1}{\line(-7,4){5}}

\allinethickness{2pt}
\put(30,4.3){\line(-4,7){19.9}}
\put(30,4.4){\line(-7,4){30}}
\put(0,4.3){\line(1,0){35}}
\put(30,4.3){\line(-7,-4){7.6}}
\put(30,4.3){\line(7,4){5}}
\put(30,4.3){\line(7,-4){5}}
\put(30,4.3){\line(0,1){34.8}}
\put(30,4.3){\line(0,-1){4.4}}
\put(30,4.3){\line(4,7){5}}
\put(30,4.3){\line(4,-7){2.5}}
\put(30,4.3){\line(-4,-7){2.5}}

  \linethickness{0.075mm}
  \multiput(0,0)(2.5,0){15}{\line(0,1){39.1}}
  \multiput(0,0)(0,4.35){10}{\line(1,0){35}}
\end{picture}}
\caption{Extra $H^2$-nonvanishing in the chamber $s$}
\end{center}
\end{figure}
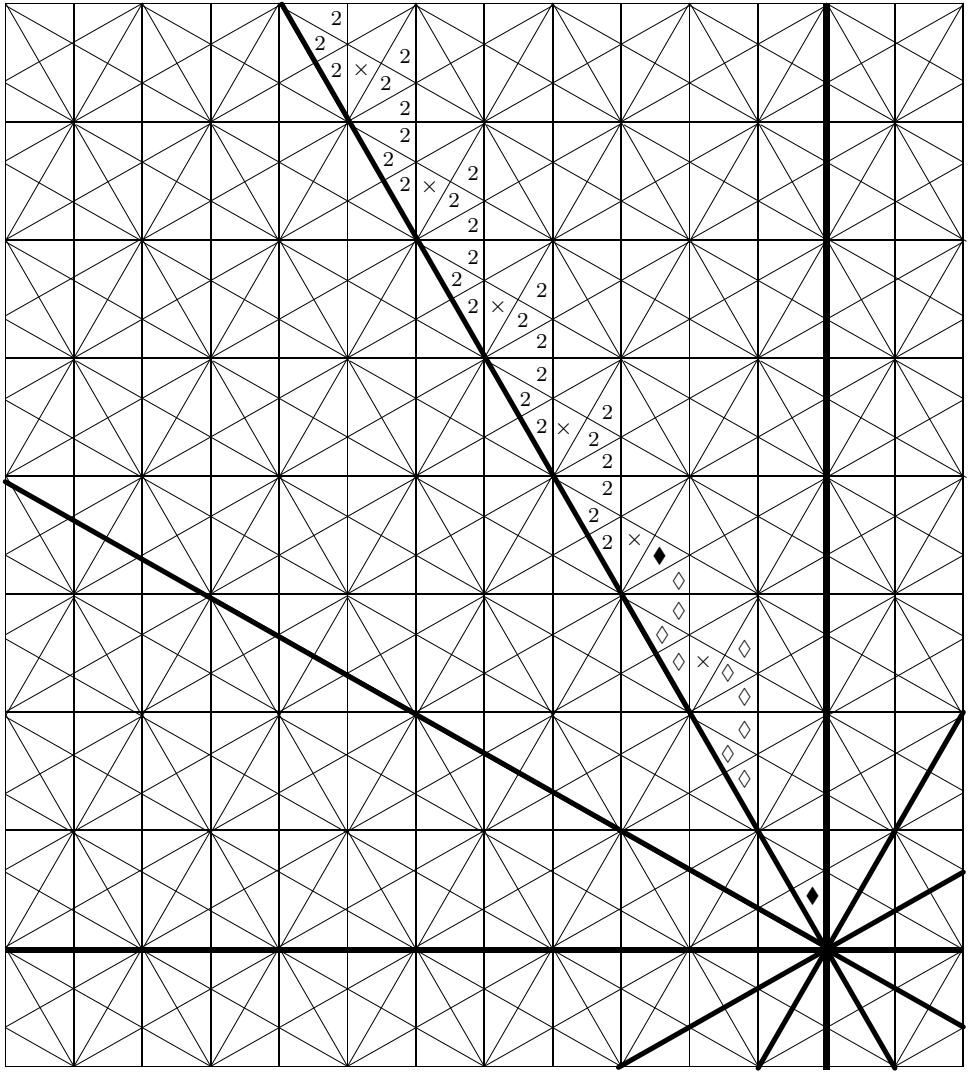

 Consider the $\beta$-sequences for $A^s_{18}$
 \[ H^2(V) \to H^3(A_{18}^y) \to H^2(A_{18}^s) \]
 and
 \[ H^2(C) \to H^2(V) \to H^1(Q).\] 
 By our results 
obtained  in 5.1 we have $H^3(A_{18}^y) \neq 0$. Moreover, $C = Q = A_{18}^y + p \beta = A_{10}$ is dominant so that $H^j(C) = H^j(Q) = 0 $ for all $j>0$. The non-vanishing
 of  $H^2(A_{18}^s)$ follows.
 
 To see that $H^2(A^s_{33})$ is non-zero, we consider the $\alpha$-sequences
 \[ H^0(A_{33} ) \to H^1(V) \to H^2(A_{33}^s) \]
 and 
 \[ H^1(V) \to H^0(Q) \to H^2(C). \]
 The weights of $C$ and $Q$ are $A^s_{33} + p \alpha = A^s_{26}$ and $A^s_{33} + 2p \alpha = A_{27}$. It follows that $H^2(C) = 0$ by Proposition 3.5.i)
 and by Kempf,
 and that we have an exact sequence
 \[ 
 H^0(Q) \to H^0(A_{27}) \to H^1(A^s_{26}).\]
 Here the last module is isomorphic to $H^0(A_{26})$
 by Proposition 3.5.i) again. According to \cite{A86} $L(A_6)$ is a composition factor of $H^0(A_{27})$ but not of $H^0(A_{26})$. We conclude that
 $L(A_6)$ is therefore a composition factor of $H^0(Q)$ and hence also of $H^1(V)$. As it is not (\cite{A86} again) a composition factor of $H^0(A_{33})$ the first sequence gives
 the claim.
 Although $A_{26}, A_{27}, A_{33}$
 do not live in the Jantzen region for $p=7$,
 one can  compute 
as in 5.2
 the relevant multiplicities of $L(A_6)$.

 Finally, to check the vanishing of $H^2(A^s_{34})$ we look at the $\beta$-sequences
 \[ H^3(A^y_{34}) \to H^2(A_{34}^s) \to H^3(V) \]
 and \[ H^3(C) \to H^3(V) \to H^2(Q). \]
 Here $H^3(A^y_{34}) = 0$ 
 by 4.1
 and considering the weights $A^y_{34} + 2p\beta = A_{19}$ and $A^y_{34} + p\beta = A^t_{15}$, we 
 get $H^2(Q) = 0 = H^3(C)$
 from 5.1 and 4.1. 

 \section{The remaining cases;
 what should we name?}
 The above results together with Serre duality describe the vanishing behavior of all line bundle cohomology corresponding to weights in the lowest $p^2$ alcoves.
 The results
 are collected in Figure 12 below (identical to Fig. 3 in \cite[\S8]{AK10}).
 
 \begin{figure}
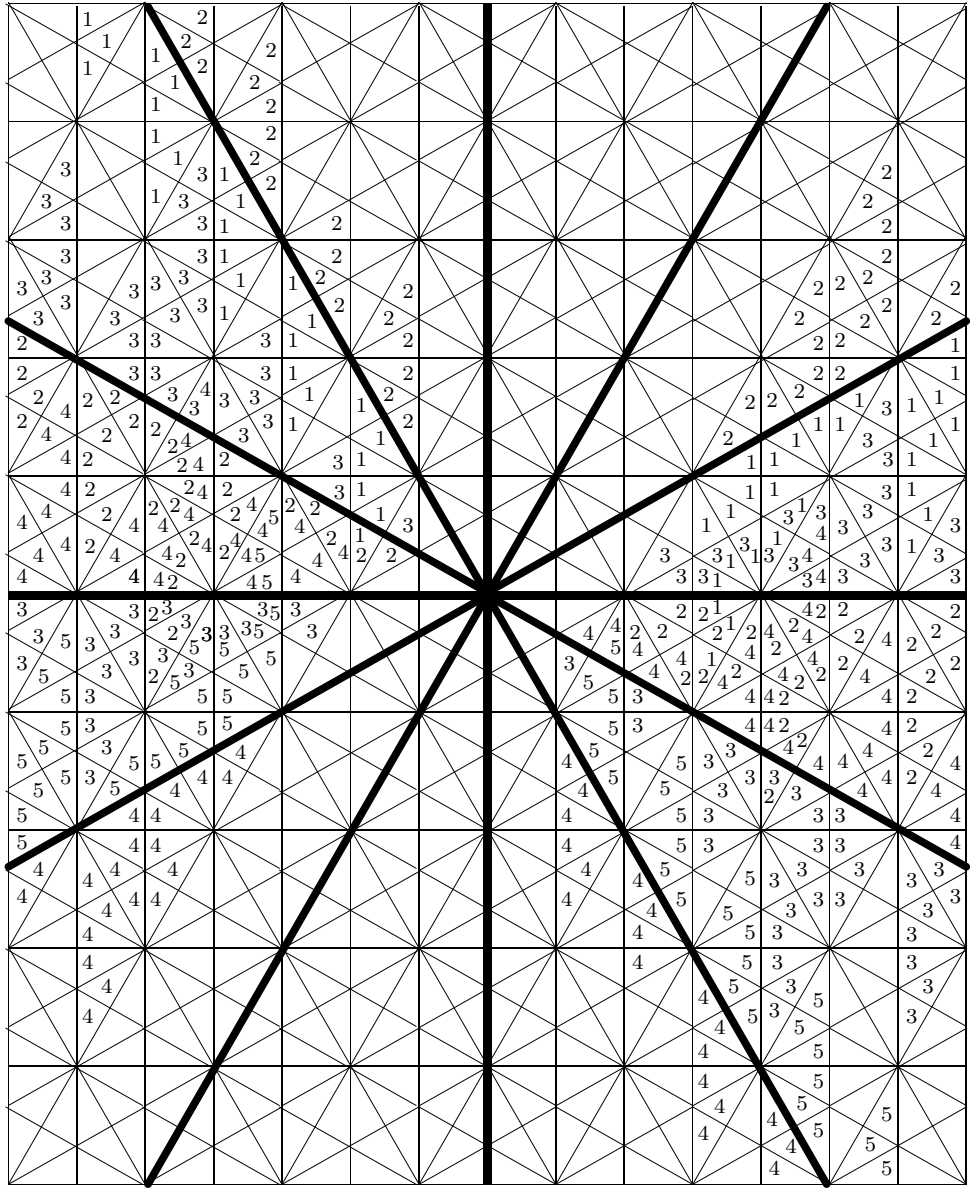

\begin{center}
{\tiny%
  \setlength{\unitlength}{3.6mm}
}

\caption{Extra cohomology non-vanishing
for $p\geq11$}
\end{center}
\end{figure}

\section{Higher alcoves}
In \cite{A81} (5.7) the first author formulated a principle which he hoped would allow to extrapolate the vanishing behavior of 
$H^i$ in the higher alcoves from those in the closure  $\overline P_2$ 
of the union
of the lowest $p^2$-alcoves. Properly formulated (see \cite{AK10} Appendix B) this principle says

{\it For an arbitrary weight $ \lambda$ we have $H^i(\lambda) \neq 0$ if and only if there exists $ \chi \in \overline P_2$ with $H^i(\chi) \neq 0$ 
such that $\lambda \in p^n \cdot \chi \pm X_{p^n}$  for some  $n$.} 

The `if part'  of this principle 
follows from Remark 3.1.
The `only if part' has only been checked (using the methods above) 
for types $A_2$ and $B_2$ (and in general for $i= 1$ and $i = N-1$). The following example shows that{ \bf it fails for type $G_2$.}   

Consider $A = (2p^2 - p) \omega_\alpha - p^2 \omega_\beta + A_2$. This alcove is a subset of $p\cdot \nu +X_p$ and of $p\cdot \nu' - X_p$ where $\nu = 
2(p-1) \omega_\alpha - (p+1) \omega_\beta$ and $\nu' = \nu + \rho$. Note that $\nu \in F^w_{3/4}$ and $\nu' \in F^y_{4/5}$ so that by our results in Section 4.1 we have $H^4(\nu) = 0
= H^4(\nu')$. Therefore the principle above predicts vanishing of $H^4(A)$. However, if $A'$ is the alcove obtained by reflecting $A$ in its mid-sized wall then 
$H^5(A') \neq 0$ (e.g., because $A' \subset p^2 \cdot (-3\omega_\beta) + X_{p^2}$). It follows from Proposition 3.4 that then also $H^4(A) \neq 0$ and hence (5.7) fails for weights in $A$.
Using the arguments in Proposition 3.2.ii) we find many ($17$ for $p=7$) other alcoves for which (5.7) fails.

We shall illustrate this result by taking $p=7$.
Enter first $5$'s in the alcoves with extra $H^5$-non-vanishing in Fig. 10.
Then enter $4$'s where we have extra $H^4$-non-vanishing in
$P_2$ from \S4.
Also enter $4$'s in  the alcoves
belonging to those $p^2$-alcoves where Proposition 3.1 gives extra non-vanishing $H^4$.

Let
$A_i^2$ denote the $p^2$-alcove of type
$i$
and let
$A_i^{2,v}=v\cdot A_i^2$, $v\in W$.
By Proposition 3.4 we have non-vanishing $H^4$ in all of $A_4^{2,w}$ and we therefore mark all alcoves here which 
are not already equipped with 
a $4$ by $\lozenge$.
Let $A'$ be the top alcove of $A_4^{2,w}$
and let
$A=A't$.
From (3.1)
there is an exact sequence
$H^4(A)\to H^5(A')\to \theta_tH^5(A)$.
As
$H^5(A')\ne0=\theta_tH^5(A)$, we get
$H^4(A)\ne0$.
Mark thus $A$ with $\lozenge$,
and also
all alcoves
in $A_3^{2,w}$ below $A$
by Proposition 3.2.iia) (unless such an alcove already contain a $4$).
By Proposition 3.3
we can mark 2 translates of $A_1$ with
$\times$, and then 
$H^4$ vanishes on all alcoves above either of those;
the relevant special points are
$p\cdot(p\omega_\alpha-(p-1)\omega_\beta)$
and
$p\cdot((p+3)\omega_\alpha-p\omega_\beta)$
with both
$p\omega_\alpha-(p-1)\omega_\beta$ and
$(p+3)\omega_\alpha-p\omega_\beta\in
F^w_{2/3}$.

Now 
\begin{align*}
A
&=(2p-1)p\omega_\alpha-p^2\omega_\beta+A_2
\\
&\subseteq
\{p\cdot((2p-2)\omega_\alpha-(p+1)\omega_\beta)+X_p\}
\cap
\\
&\hspace{5cm}
\{p\cdot((2p-1)\omega_\alpha-p\omega_\beta)-X_p\}
\end{align*}
with
$(2p-2)\omega_\alpha-(p+1)\omega_\beta \in
F_{3/4}^w$
and
$(2p-1)\omega_\alpha-p\omega_\beta \in F_{4/5}^y$.
But our results in $P_2$ say that 
$H^4(F_{3/4}^w)=0=H^4(F_{4/5}^y)$.
Also,
\begin{align*}
A
&\subseteq
\{
p^2\cdot(-2\omega_\beta)+X_{p^2}\}\cap
\{
p^2\cdot(\omega_\alpha-\omega_\beta)-X_{p^2}\}
\cap
\\
&\cap_{r\geq3}
\{
\{
p^r\cdot(-\omega_\alpha-2\omega_\beta)+X_{p^r}\}\cap
\{
p^r\cdot(-\omega_\beta)-X_{p^r}\}
\}
\end{align*}
with
$-2\omega_\beta\in F_{A_1^{sw_0}/A_1^{tz}}$,
$\omega_\alpha-\omega_\beta\in F_{A_1/A_1^{t}},
-\omega_\beta\in F_{A_1/A_1^{t}}$
and
$-\omega_\alpha-2\omega_\beta\in F_{A_1^{w_0}/A_1^{sw_0}}$,
showing failure of the hoped-for principle.
Moreover,
$H^4(A-p\omega_\alpha)\ne0$
with
$
A-p\omega_\alpha
\subseteq
\{p\cdot((2p-3)\omega_\alpha-(p+1)\omega_\beta)+X_p\}
\cap
\{p\cdot((2p-2)\omega_\alpha-p\omega_\beta)-X_p\}
$
and
$(2p-3)\omega_\alpha-(p+1)\omega_\beta\in F_{3/4}^w$,
$(2p-2)\omega_\alpha-p\omega_\beta\in
F_{3/4}^y$.
Likewise
$H^4(A-2p\omega_\alpha)\ne0$
with
$
A-2p\omega_\alpha
\subseteq
\{p\cdot((2p-4)\omega_\alpha-(p+1)\omega_\beta)+X_p\}
\cap
\{p\cdot((2p-3)\omega_\alpha-p\omega_\beta)-X_p\}
$
and
$(2p-4)\omega_\alpha-(p+1)\omega_\beta\in F_{3/4}^w$,
$(2p-3)\omega_\alpha-p\omega_\beta\in
F_{A_2^w/A_2^y}$,
$H^4(F_{A_2^w/A_2^y})=0$.
Also,
\begin{align*}
A-p\omega_\alpha, 
&
A-2p\omega_\alpha
\subseteq
\{
p^2\cdot(-2\omega_\beta)+X_{p^2}\}\cap
\{
p^2\cdot(\omega_\alpha-\omega_\beta)-X_{p^2}\}
\cap
\\
&\cap_{r\geq3}
\{
\{
p^r\cdot(-\omega_\alpha-2\omega_\beta)+X_{p^r}\}\cap
\{
p^r\cdot(-\omega_\beta)-X_{p^r}\}
\}.
\end{align*}
Thus all the alcoves in
$A_3^{2,w}$ above
$A_1+(2p-3)p\omega_\alpha-p^2\omega_\beta=(A-2p\omega_\alpha)s$
with $s$ the short wall, 18 of them together in this $p=7$ case,
exhibit failure of the principle.

On the other hand,
let
$\nu_1=p\cdot((2p-5)\omega_\alpha-p\omega_\beta)$.
As
$(2p-5)\omega_\alpha-p\omega_\beta\in A_3^w$ with
$H^4(A_3^w)\ne0$,
the alcoves in $\nu_1+X_p$
are marked $\lozenge$ by the
`if' part of the principle.
If $\nu_2=\nu_1-p\omega_\alpha=
p\cdot((2p-6)\omega_\alpha-p\omega_\beta)$,
$(2p-6)\omega_\alpha-p\omega_\beta\in
F_{A_3^w/A_3^{tz}}$
with
$H^4(F_{A_3^w/A_3^{tz}})\ne0$.
Thus the alcoves in $\nu_2+X_p$
are also marked with $\lozenge$.


\begin{figure}
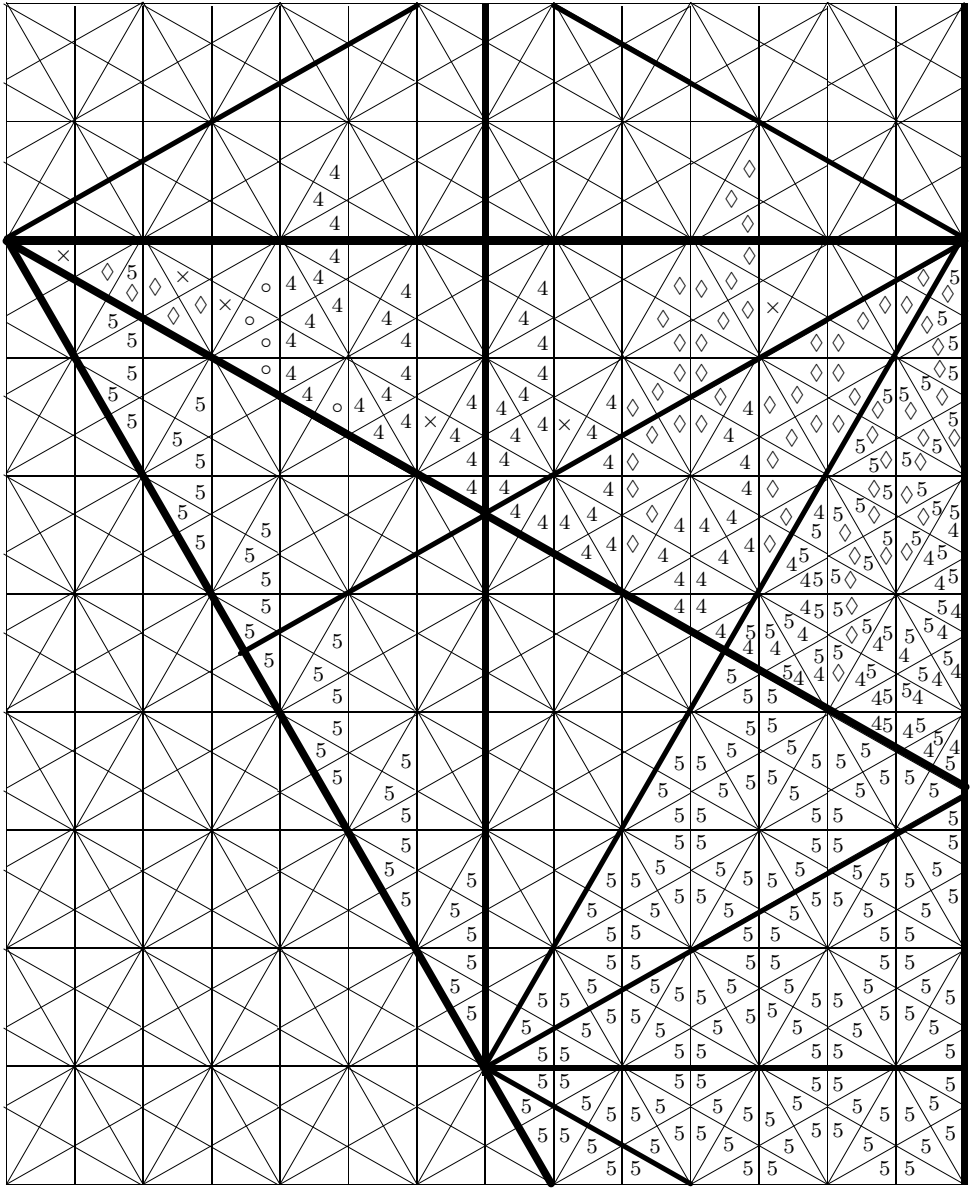

\begin{center}
{\tiny%
  \setlength{\unitlength}{3.6mm}
}

\caption{Beyond $p^2$-alcoves for $p=7$}
\end{center}
\end{figure}

\section{The quantum case}
As in the introduction we let $U_q$ denote the quantum group of type $G_2$ at a root of unity $q\in \C$. More precisely,
we start with the generic quantum group $U_v$ of type $G_2$ over the field $\Q(v)$, $v$ being an indeterminate. The generators in $U_v$ are denoted $F_i, E_i, K_i^{\pm}$,
$i = 1, 2$. We choose the enumeration such that the indices $1$ and $2$ correspond in the notation of the previous sections to the short simple root $\alpha$ and the 
long simple root $\beta$, respectively. 
Then we set $A = \Z[v, v^{-1}]$ and let $U_A$ denote the Lusztig $A$-form in $U_v$, see e.g. \cite {L90}. Considering $\C$ as an
$A$-algebra via the specialization $v \mapsto q$ we then define $U_q = U_A \otimes_A \C$. This construction makes sense for all non-zero $q \in \C$ but for our purpose the
interesting case is when $q$ is a root of unity. So {\em for us $q \in \C$ will always denote a primitive root of $1$ of order $l>5$ and we assume $(l,6) = 1$}.

We have a triangular decomposition $U_v = U_v^-U_v^0U_v^+$ with $U_v^-$, $U_v^0$ and $U_v^+$, being the subalgebras generated by the $F$'s, $K$'s, and $E$'s, respectively.
There is a corresponding decomposition $U_q = U_q^-U_q^0U_q^+$. Following the convention in \cite{APW1} we set $B_q = U_q^-U_q^0$ and denote by $H^0_q$ the induction functor 
from integrable $B_q$-modules to integrable $U_q$-modules. This is a
left exact functor and we denote its $i$-th right derived functor by $H_q^i$, $i\in \N$. Setting $X = \Z^2$ each $\lambda \in X$ defines a $1$-dimensional $B_q$-module which we
also denote $\lambda$ and then the $H^i_q(\lambda)$'s are quantized counterparts of the line bundle cohomology studied in the previous sections. For details we refer to
\cite{APW1}.

Now all the methods and techniques from Sections 2-3 apply in this case as well. In addition to \cite{APW1} we refer to \cite{APW2},
\cite{AW}
which
we need for non prime power $l$, and \cite{A03}. These allow us to describe
completely the vanishing behavior for $H^i_q(\lambda)$ for all $i\in \N$ and all $\lambda \in X$. The results are completely the same as those described in Figures 2-11 with 
the only
difference being that $p$ should be replaced by $l$ and there are no upper bounds on $\lambda$ (i.e.,
the condition $\lambda \in P_2$ is not relevant in the quantum case).
So we can summarize this as

\begin{prop}
The vanishing behavior of the cohomology modules $H^i_q(\lambda)$ 
of the quantum algebra
is given by Fig. 9 with all alcoves being $l$-alcoves and with the figure extended to the whole plane.
\end{prop}

\end{document}